\documentclass{amsart}
\usepackage{amssymb}

\newtheorem{theorem}{Theorem}[section]
\newtheorem{lemma}[theorem]{Lemma}

\newtheorem{corollary}[theorem]{Corollary}
\theoremstyle{definition}
\newtheorem{definition}[theorem]{Definition}

\newtheorem{conj}[theorem]{Conjecture}

\theoremstyle{remark}

\numberwithin{equation}{section}

\begin{document}

\title{Counting Points on Dwork Hypersurfaces and $p$-adic Gamma Function}


\author{Rupam Barman}
\address{Department of Mathematics, Indian Institute of Technology, Hauz Khas, New Delhi-110016, INDIA}
\curraddr{}
\email{rupam@maths.iitd.ac.in}

\author{Hasanur Rahman}
\address{Department of Mathematics, Indian Institute of Technology, Hauz Khas, New Delhi-110016, INDIA}
\curraddr{}
\email{hasrah93@gmail.com}
\thanks{}
\author{Neelam Saikia}
\address{Department of Mathematics, Indian Institute of Technology, Hauz Khas, New Delhi-110016, INDIA}
\curraddr{}
\email{nlmsaikia1@gmail.com}
\thanks{}

\date{29th November, 2015}
\thanks{We are grateful to Wadim Zudilin and Dermot McCarthy for many helpful comments on an initial draft of the manuscript.}
\begin{abstract}
We express the number of points on the Dwork hypersurface
$$X_{\lambda}^d: x_1^d+x_2^d+\cdots +x_d^d=d\lambda x_1x_2\cdots x_d$$
over a finite field of order $q \not \equiv 1 \pmod{d}$ in terms of McCarthy's $p$-adic hypergeometric function for any odd prime $d$. 
\end{abstract}
\maketitle
\section{Introduction and statement of results}
Let $p$ be an odd prime, and let $\mathbb{F}_q$ denote the finite field with $q$ elements, where $q=p^r, r\geq 1$.
Let $\mathbb{Z}_p$ denote the ring of $p$-adic integers.
Let $\Gamma_p(\cdot)$ denote the Morita's $p$-adic gamma function, and let $\omega$ denote the
Teichm\"{u}ller character of $\mathbb{F}_q$. We denote by $\overline{\omega}$ the inverse of $\omega$.
For $x \in \mathbb{Q}$ we let $\lfloor x\rfloor$ denote the greatest integer less than
or equal to $x$ and $\langle x\rangle$ denote the fractional part of $x$, i.e., $x-\lfloor x\rfloor$.
Also, we denote by $\mathbb{Z}^{+}$ and $\mathbb{Z}_{\geq 0}$
the set of positive integers and non negative integers, respectively. We now define the McCarthy's $p$-adic hypergeometric series $_{n}G_{n}[\cdots]$
as follows.
\begin{definition}\cite[Definition 5.1]{mccarthy2} \label{defin1}
Let $q=p^r$, for $p$ an odd prime and $r \in \mathbb{Z}^+$, and let $t \in \mathbb{F}_q$.
For $n \in \mathbb{Z}^+$ and $1\leq i\leq n$, let $a_i$, $b_i$ $\in \mathbb{Q}\cap \mathbb{Z}_p$.
Then the function $_{n}G_{n}[\cdots]$ is defined by
\begin{align}
&_nG_n\left[\begin{array}{cccc}
             a_1, & a_2, & \ldots, & a_n \\
             b_1, & b_2, & \ldots, & b_n
           \end{array}|t
 \right]_q:=\frac{-1}{q-1}\sum_{j=0}^{q-2}(-1)^{jn}~~\overline{\omega}^j(t)\notag\\
&\times \prod_{i=1}^n\prod_{k=0}^{r-1}(-p)^{-\lfloor \langle a_ip^k \rangle-\frac{jp^k}{q-1} \rfloor -\lfloor\langle -b_ip^k \rangle +\frac{jp^k}{q-1}\rfloor}
 \frac{\Gamma_p(\langle (a_i-\frac{j}{q-1})p^k\rangle)}{\Gamma_p(\langle a_ip^k \rangle)}
 \frac{\Gamma_p(\langle (-b_i+\frac{j}{q-1})p^k \rangle)}{\Gamma_p(\langle -b_ip^k \rangle)}.\notag
\end{align}
\end{definition}
\par
Koblitz \cite{kob2} developed a formula for the number of points on diagonal hypersurfaces in the Dwork family in terms of Gauss sums. In \cite{goodson}, H. Goodson specializes Koblitz's formula to the family of Dwork K3 surfaces. She gives an expression for the number of points on the family of Dwork K3 surfaces $X_{\lambda}^4: x_1^4+x_2^4+x_3^4+x_4^4=4\lambda x_1x_2x_3x_4$ in the projective plane $\mathbb{P}^4(\mathbb{F}_q)$ over a finite field $\mathbb{F}_q$ in terms of Greene's finite field hypergeometric functions \cite{greene} under the condition that $q\equiv 1 \pmod{4}$. 
She then considers the higher dimensional Dwork hypersurfaces $$X_{\lambda}^d: x_1^d+x_2^d+\cdots +x_d^d=d\lambda x_1x_2\cdots x_d.$$
She gives a formula for the number of points on $X_{\lambda}^d$ in terms of Gaussian hypergeometric series and Gauss sums when $q\equiv 1 \pmod{d}$. For primes $p\not \equiv 1 \pmod{d}$, she conjectures the following.
\begin{conj}\cite[Conjecture 8.2]{goodson}
 Let $d$ be an odd prime and $p$ a prime number such that $p\not \equiv 1 \pmod{d}$. The number of points over $\mathbb{F}_p$ on the Dwork hypersurface is given by
 \begin{align}
\#X_{\lambda}^d(\mathbb{F}_p)=\frac{p^{d-1}-1}{p-1}+\frac{1}{p-1}+{_{d-1}G}_{d-1}\left[\begin{array}{cccc}
                                                             \frac{1}{d}, & \frac{2}{d}, & \ldots, & \frac{d-1}{d} \\
                                                             0, & 0, & \ldots, & 0
                                                           \end{array}|\lambda^d
\right]_p.\notag
 \end{align}
\end{conj}
In this article we prove that the above conjecture is not correct. We correct the statement of the conjecture and prove it for any finite field of order 
$q=p^r\not \equiv 1 \pmod{d}$. The statement of the main result is as follows.
\begin{theorem}\label{MT_1}
Let $d$ be an odd prime and $q=p^r$ be a prime power such that $q\not\equiv 1 \pmod{d}$ and $p\neq d$. Then the number of points on the Dwork hypersurface $$X_{\lambda}^d: x_1^d+x_2^d+\cdots+x_d^d=d\lambda x_1x_2\cdots x_d$$ 
in $\mathbb{P}^d(\mathbb{F}_q)$ is given by
\begin{align}
\#X_{\lambda}^d(\mathbb{F}_q)=\frac{q^{d-1}-1}{q-1}-{_{d-1}G}_{d-1}\left[\begin{array}{cccc}
                                                             \frac{1}{d}, & \frac{2}{d}, & \ldots, & \frac{d-1}{d} \\
                                                             0, & 0, & \ldots, & 0
                                                           \end{array}|\lambda^d
\right]_q.\notag
\end{align}
\end{theorem}
We use the technique of D. McCarthy \cite{dermot2} to prove the above theorem, and the case $d=5$ is dealt by McCarthy in the same paper.
We note that the expression in the above conjecture contains an error term as seen in the above theorem. Also, the sign of the $G$-function is negative when $d$ is an odd prime.
\par For any $\lambda$ and $q=p^r \not\equiv 1 \pmod{3}$, we have 
$$\#X_{\lambda}^3(\mathbb{F}_q)=1+\# \{(x, y)\in \mathbb{F}_q^2: x^3+y^3+1=3\lambda xy\}.$$
Now, from Therem \ref{MT_1} and Theorem 3.3 of \cite{BS2}, we have the following transformation for the $_{2}G_2$-function. 
\begin{corollary}
 Let $\lambda \neq 0$ and $\lambda^3\neq 1$. Let $p\geq 5$ be a prime and $q=p^r \not\equiv 1 \pmod{3}$. Then
 $${_{2}G}_{2}\left[\begin{array}{cc}
                                                             \frac{1}{3}, & \frac{2}{3} \\
                                                             0, & 0 
                                                             \end{array}|\lambda^3
\right]_q=q\phi(-3\lambda){_{2}G}_{2}\left[\begin{array}{cc}
                                                             \frac{1}{2}, & \frac{1}{2} \\
                                                             \frac{1}{6}, & \frac{5}{6} 
                                                             \end{array}|\frac{1}{\lambda^3}
\right]_q,$$
where $\phi$ is the quadratic character on $\mathbb{F}_q$.
\end{corollary}
\section{Preliminaries}
Let $\widehat{\mathbb{F}_q^\times}$ denote the set of all multiplicative characters $\chi$ on $\mathbb{F}_q^{\times}$.
It is known that $\widehat{\mathbb{F}_q^\times}$ is a cyclic group of order $q-1$
under the multiplication of characters: $(\chi\psi)(x)=\chi(x)\psi(x)$, $x\in \mathbb{F}_q^{\times}$.
The domain of each
$\chi \in \widehat{\mathbb{F}_q^{\times}}$ is extended to $\mathbb{F}_q$ by setting $\chi(0):=0$ including the trivial character $\varepsilon$.
Multiplicative characters satisfy the following \emph{orthogonality relations}.
\begin{lemma}\emph{(\cite[Chapter 8]{ireland}).}\label{lemma2} We have
\begin{enumerate}
\item $\displaystyle\sum_{x\in\mathbb{F}_q}\chi(x)=\left\{
                                  \begin{array}{ll}
                                    q-1 & \hbox{if~ $\chi=\varepsilon$;} \\
                                    0 & \hbox{if ~~$\chi\neq\varepsilon$.}
                                  \end{array}
                                \right.$
\item $\displaystyle\sum_{\chi\in \widehat{\mathbb{F}_q^\times}}\chi(x)~~=\left\{
                            \begin{array}{ll}
                              q-1 & \hbox{if~~ $x=1$;} \\
                              0 & \hbox{if ~~$x\neq1$.}
                            \end{array}
                          \right.$
\end{enumerate}
\end{lemma}
\par Let $\mathbb{Z}_p$ and $\mathbb{Q}_p$ denote the ring of $p$-adic integers and the field of $p$-adic numbers, respectively.
Let $\overline{\mathbb{Q}_p}$ be the algebraic closure of $\mathbb{Q}_p$ and $\mathbb{C}_p$ the completion of $\overline{\mathbb{Q}_p}$.
Let $\mathbb{Z}_q$ be the ring of integers in the unique unramified extension of $\mathbb{Q}_p$ with residue field $\mathbb{F}_q$.
We know that $\chi\in \widehat{\mathbb{F}_q^{\times}}$ takes values in the group of
$(q-1)$-th roots of unity in $\mathbb{C}^{\times}$. Since $\mathbb{Z}_q^{\times}$ contains all $(q-1)$-th roots of unity,
we can consider multiplicative characters on $\mathbb{F}_q^\times$
to be maps $\chi: \mathbb{F}_q^{\times} \rightarrow \mathbb{Z}_q^{\times}$.
Let $\omega: \mathbb{F}_q^\times \rightarrow \mathbb{Z}_q^{\times}$ be the Teichm\"{u}ller character.
For $a\in\mathbb{F}_q^\times$, the value $\omega(a)$ is just the $(q-1)$-th root of unity in $\mathbb{Z}_q$ such that $\omega(a)\equiv a \pmod{p}$.
\par We now introduce some properties of Gauss sums. For further details, see \cite{evans}. Let $\zeta_p$ be a fixed primitive $p$-th root of unity
in $\overline{\mathbb{Q}_p}$. The trace map $\text{tr}: \mathbb{F}_q \rightarrow \mathbb{F}_p$ is given by
\begin{align}
\text{tr}(\alpha)=\alpha + \alpha^p + \alpha^{p^2}+ \cdots + \alpha^{p^{r-1}}.\notag
\end{align}
Then the additive character
$\theta: \mathbb{F}_q \rightarrow \mathbb{Q}_p(\zeta_p)$ is defined by
\begin{align}
\theta(\alpha)=\zeta_p^{\text{tr}(\alpha)}.\notag
\end{align}
For $\chi \in \widehat{\mathbb{F}_q^\times}$, the \emph{Gauss sum} is defined by
\begin{align}
g(\chi):=\sum_{x\in \mathbb{F}_q}\chi(x)\theta(x).\notag
\end{align}
We let $T$ denote a fixed generator of $\widehat{\mathbb{F}_q^\times}$.
We now state two important properties of Gauss sums.
\begin{lemma}\emph{(\cite[Eqn. 1.12]{greene}).}\label{fusi3}
If $k\in\mathbb{Z}$ and $T^k\neq\varepsilon$, then
$$g(T^k)g(T^{-k})=qT^k(-1).$$
\end{lemma}
\begin{lemma}\emph{(\cite[Lemma 2.2]{Fuselier}).}\label{lemma1}
For all $\alpha \in \mathbb{F}_q^{\times}$, $$\theta(\alpha)=\frac{1}{q-1}\sum_{m=0}^{q-2}g(T^{-m})T^m(\alpha).$$
\end{lemma}
\par
Finally, we recall the $p$-adic gamma function. For further details, see \cite{kob}.
For $n \in\mathbb{Z}^+$,
the $p$-adic gamma function $\Gamma_p(n)$ is defined as
\begin{align}
\Gamma_p(n):=(-1)^n\prod_{0<j<n,p\nmid j}j\notag
\end{align}
and one extends it to all $x\in\mathbb{Z}_p$ by setting $\Gamma_p(0):=1$ and
\begin{align}
\Gamma_p(x):=\lim_{n\rightarrow x}\Gamma_p(n)\notag
\end{align}
for $x\neq0$, where $n$ runs through any sequence of positive integers $p$-adically approaching $x$.
This limit exists, is independent of how $n$ approaches $x$,
and determines a continuous function on $\mathbb{Z}_p$ with values in $\mathbb{Z}_p^{\times}$.
Let $\pi \in \mathbb{C}_p$ be the fixed root of $x^{p-1} + p=0$ which satisfies
$\pi \equiv \zeta_p-1 \pmod{(\zeta_p-1)^2}$. Then the Gross-Koblitz formula relates Gauss sums and the $p$-adic gamma function as follows.
Recall that $\omega$ denotes the Teichm\"{u}ller character of $\mathbb{F}_q$.
\begin{theorem}\emph{(\cite[Gross-Koblitz]{gross}).}\label{thm4} For $a\in \mathbb{Z}$ and $q=p^r$,
\begin{align}
g(\overline{\omega}^a)=-\pi^{(p-1)\sum_{i=0}^{r-1}\langle\frac{ap^i}{q-1} \rangle}\prod_{i=0}^{r-1}\Gamma_p\left(\left\langle \frac{ap^i}{q-1} \right\rangle\right).\notag
\end{align}
\end{theorem}
\section{Proof of Theorem \ref{MT_1}}
We first state two lemmas which we will use to prove the theorem. The first lemma is a generalization of Lemma 4.1 in \cite{mccarthy2}.
For a proof, see \cite{BS1}.
\begin{lemma}\emph{(\cite[Lemma 3.1]{BS1}).}\label{lemma4}
Let $p$ be a prime and $q=p^r$. For $0\leq j\leq q-2$ and $t\in \mathbb{Z^+}$ with $p\nmid t$, we have
\begin{align}\label{eq8}
\omega(t^{tj})\prod_{i=0}^{r-1}\Gamma_p\left(\left\langle \frac{tp^ij}{q-1}\right\rangle\right)
\prod_{h=1}^{t-1}\Gamma_p\left(\left\langle\frac{hp^i}{t}\right\rangle\right)
=\prod_{i=0}^{r-1}\prod_{h=0}^{t-1}\Gamma_p\left(\left\langle\frac{p^ih}{t}+\frac{p^ij}{q-1}\right\rangle\right),
\end{align}
\begin{align}\label{eq9}
\omega(t^{-tj})\prod_{i=0}^{r-1}\Gamma_p\left(\left\langle\frac{-tp^ij}{q-1}\right\rangle\right)
\prod_{h=1}^{t-1}\Gamma_p\left(\left\langle \frac{hp^i}{t}\right\rangle\right)
=\prod_{i=0}^{r-1}\prod_{h=0}^{t-1}\Gamma_p\left(\left\langle\frac{p^i(1+h)}{t}-\frac{p^ij}{q-1}\right\rangle \right).
\end{align}
\end{lemma}
\begin{lemma}\label{lemma5}
Let $d\neq p$ be a prime number such that $q=p^r\not\equiv1\pmod{d}$. Then, for $1\leq a\leq q-2$  and $0\leq i\leq r-1$ we have
\begin{align}\label{eq-51}
d\left\lfloor\frac{ap^i}{q-1}\right\rfloor+\left\lfloor\frac{-dap^i}{q-1}\right\rfloor
&=(d-1)\left\lfloor\frac{ap^i}{q-1}\right\rfloor+\sum_{h=1}^{d-1}\left\lfloor\left\langle\frac{hp^i}{d}\right\rangle-\frac{ap^i}{q-1}\right\rfloor-1.
\end{align}
\end{lemma}
\begin{proof}
Let $\lfloor\frac{-dap^i}{q-1}\rfloor=dk+s$ for some $k, s \in \mathbb{Z}$ satisfying $0\leq s\leq d-1$.
Since $1\leq a\leq q-2$ and $(q-1, dp^i)=1$, we observe that $\frac{-dap^i}{q-1}$ is not an integer. This yields
\begin{align}
dk+s<\frac{-dap^i}{q-1}<dk+s+1,\notag
\end{align}
which implies
\begin{align}\label{eq-7}
\frac{s}{d}+k<\frac{-ap^i}{q-1}<k+\frac{s+1}{d}.
\end{align}
This is equivalent to
\begin{align}\label{eq-8}
-\frac{s+1}{d}-k<\frac{ap^i}{q-1}<-k-\frac{s}{d}.
\end{align}
From \eqref{eq-8} we have $\lfloor\frac{ap^i}{q-1}\rfloor=-k-1$, and then the left hand side of \eqref{eq-51} becomes $s-d$. 
Again, since $d$ is a prime and $d\neq p$, we observe that 
\begin{align}\label{eq-9}
\sum_{h=1}^{d-1}\left\lfloor\left\langle\frac{hp^i}{d}\right\rangle-\frac{ap^i}{q-1}\right\rfloor
=\sum_{h=1}^{d-1}\left\lfloor\left\langle\frac{h}{d}\right\rangle-\frac{ap^i}{q-1}\right\rfloor.
\end{align}
Thus, for $1\leq h\leq d-s-1$, \eqref{eq-7} yields 
$$\left\lfloor\left\langle\frac{h}{d}\right\rangle-\frac{ap^i}{q-1}\right\rfloor=k$$ so that
\begin{align}\label{eq-10}
\sum_{h=1}^{d-s-1}\left\lfloor\left\langle\frac{h}{d}\right\rangle-\frac{ap^i}{q-1}\right\rfloor=(d-s-1)k.
\end{align}
Also, for $d-s\leq h\leq d-1$, \eqref{eq-7} yields 
$$\left\lfloor\left\langle\frac{h}{d}\right\rangle-\frac{ap^i}{q-1}\right\rfloor=k+1$$ so that
\begin{align}\label{eq-11}
\sum_{h=d-s}^{d-1}\left\lfloor\left\langle\frac{h}{d}\right\rangle-\frac{ap^i}{q-1}\right\rfloor=s(k+1).
\end{align}
Combining \eqref{eq-10}, \eqref{eq-11} and using the fact that $\lfloor\frac{ap^i}{q-1}\rfloor=-k-1$ we obtain that the right hand side of \eqref{eq-51} also becomes $s-d$. This completes the proof of the lemma.
\end{proof}
\noindent \textbf{Proof of Theorem \ref{MT_1}}.
Let $N_q^{A}(\lambda)$ denote the number of points on the Dwork hypersurface $X_{\lambda}^d$ in $\mathbb{A}^d(\mathbb{F}_q)$. Then
\begin{align}\label{eq-6}
\#X_\lambda^d(\mathbb{F}_q)=\frac{N_q^{A}(\lambda)-1}{q-1}.
\end{align}
Letting $\overline{x}=(x_1,x_2,\ldots,x_d)$ and $f(\overline{x})=x_1^d+x_2^d+\cdots+x_d^d-d\lambda x_1x_2\cdots x_d$ and using the identity
\begin{align}
\sum_{z\in\mathbb{F}_q}\theta(zf(\overline{x}))=\left\{
                                                  \begin{array}{ll}
                                                    q, & \hbox{if $f(\overline{x})=0$;} \\
                                                    0, & \hbox{if $f(\overline{x})\neq0$,}
                                                  \end{array}
                                                \right.\notag
\end{align} we can write
\begin{align}\label{eq-1}
q\cdot N_q^A(\lambda)&=q^d+\sum_{z\in\mathbb{F}_q^{\times}}\sum_{x_i\in\mathbb{F}_q}\theta(zf(\overline{x}))\notag\\
&=q^d+\sum_{z,x_i\in\mathbb{F}_q^{\times}}\theta(zf(\overline{x}))+\sum_{\substack{z\in\mathbb{F}_q^{\times}\\
some~x_i=0}}\theta(zf(\overline{x})).
\end{align}
We now rewrite the second summation: Let $f_1(\overline{x})=x_1^d+x_2^d+\cdots+x_d^d$ and $N_q^{\prime}$ be the number of solutions to $f_1(\overline{x})=0$ in $\mathbb{A}^d(\mathbb{F}_q)$. As we know that $x\mapsto x^d$ is an automorphism of $\mathbb{F}_q^{\times}$ when $d$ is prime and $q\not\equiv1\pmod{d}$, so $N_q^{\prime}=q^{d-1}$. Also, proceeding as above we have
\begin{align}
q\cdot N_q^{\prime}=q^d+\sum_{z,x_i\in\mathbb{F}_q^{\times}}\theta(zf_1(\overline{x}))
+\sum_{\substack{z\in\mathbb{F}_q^{\times}\\some~x_i=0}}\theta(zf_1(\overline{x})).\notag
\end{align}
Thus,
\begin{align}\label{eq-2}
\sum_{\substack{z\in\mathbb{F}_q^{\times}\\some~x_i=0}}\theta(zf_1(\overline{x}))
=-\sum_{z,x_i\in\mathbb{F}_q^{\times}}\theta(zf_1(\overline{x})).
\end{align}
Since \begin{align}\label{eq-3}
\sum_{\substack{z\in\mathbb{F}_q^{\times}\\some~x_i=0}}\theta(zf(\overline{x}))
=\sum_{\substack{z\in\mathbb{F}_q^{\times}\\some~x_i=0}}\theta(zf_1(\overline{x}))
\end{align}
so using \eqref{eq-2} and \eqref{eq-3} we can write \eqref{eq-1} as follows
\begin{align}\label{eq-4}
q\cdot N_q^A(\lambda)&=q^d+\sum_{z,x_i\in\mathbb{F}_q^{\times}}\theta(zf(\overline{x}))
-\sum_{z,x_i\in\mathbb{F}_q^{\times}}\theta(zf_1(\overline{x}))\notag\\
&=q^d+A-B,
\end{align}
where $A=\displaystyle\sum_{z,x_i\in\mathbb{F}_q^{\times}}\theta(zf(\overline{x}))$ and $B=\displaystyle\sum_{z,x_i\in\mathbb{F}_q^{\times}}\theta(zf_1(\overline{x}))$.
First we evaluate
\begin{align}
B&=\sum_{z,x_i\in\mathbb{F}_q^{\times}}\theta(zf_1(\overline{x}))\notag\\
&=\sum_{z,x_i\in\mathbb{F}_q^{\times}}\theta(zx_1^d)\theta(zx_2^d)\cdots\theta(zx_d^d).\notag
\end{align}
Lemma \ref{lemma1} gives
\begin{align}
B&=\frac{1}{(q-1)^d}\sum_{a_1,a_2,\ldots,a_d=0}^{q-2}g(T^{-a_1})g(T^{-a_2})\cdots g(T^{-a_d})\notag\\
&\times\sum_{z,x_i\in\mathbb{F}_q^{\times}}T^{a_1}(zx_1^d)T^{a_2}(zx_2^d)\cdots T(zx_d^d)\notag\\
&=\frac{1}{(q-1)^d}\sum_{a_1,a_2,\ldots,a_d=0}^{q-2}g(T^{-a_1})g(T^{-a_2})\cdots g(T^{-a_d})
\sum_{x_1\in\mathbb{F}_q^{\times}}T^{da_1}(x_1)\notag\\
&\times\sum_{x_2\in\mathbb{F}_q^{\times}}T^{da_2}(x_2)\cdots\sum_{x_d\in\mathbb{F}_q^{\times}}
T^{da_d}(x_d)\sum_{z\in\mathbb{F}_q^{\times}}T^{a_1+a_2+\cdots+a_d}(z).\notag
\end{align}
The inner sums in the above expression give non zero value only if $da_1, da_2, \ldots, da_d\equiv0\pmod{q-1}$ and $a_1+a_2+\cdots+a_d\equiv0\pmod{q-1}$. Since $q\not\equiv1\pmod{d}$ these congruences simultaneously hold only if $a_1=a_2=\cdots=a_d=0$. Finally, using the fact that $g(\varepsilon)=-1$ we obtain $B=1-q$.
Now,
\begin{align}
A&=\sum_{z,x_i\in\mathbb{F}_q^{\times}}\theta(zf(\overline{x}))\notag\\
&=\sum_{z,x_i\in\mathbb{F}_q^{\times}}\theta(zx_1^d)\theta(zx_2^d)\cdots\theta(zx_d^d)
\theta(-d\lambda zx_1x_2\cdots x_d).\notag
\end{align}
If we use Lemma \ref{lemma1} then we have
\begin{align}
A&=\frac{1}{(q-1)^{d+1}}\sum_{a_1, a_2, \ldots, a_d, a_{d+1}=0}^{q-2}g(T^{-a_1})g(T^{-a_2})\cdots
g(T^{-a_d})g(T^{-a_{d+1}})T^{a_{d+1}}(-d\lambda)\notag\\
&\times\sum_{x_1\in\mathbb{F}_q^{\times}}T^{da_1+a_{d+1}}(x_1)\sum_{x_2\in\mathbb{F}_q^{\times}}T^{da_2+a_{d+1}}(x_2)
\cdots\sum_{x_d\in\mathbb{F}_q^{\times}}T^{da_d+a_{d+1}}(x_d)\notag\\
&\times\sum_{z\in\mathbb{F}_q^{\times}}T^{a_1+a_2+\cdots a_d+a_{d+1}}(z).\notag
\end{align}
The inner sums in the above expression yield non zero value only when the following congruences hold:
$da_1+a_{d+1}, da_2+a_{d+1}, \ldots, da_d+a_{d+1}\equiv0\pmod{q-1}$ and $a_1+a_2+\cdots+a_d+a_{d+1}\equiv0\pmod{q-1}$, which implies that $a_1=a_2=\cdots=a_d=a$ (say) and $a_{d+1}=-da$ as $q\not\equiv 1 \pmod{d}$.
Thus we have
\begin{align}
A&=\sum_{a=0}^{q-2}g^d(T^{-a})g(T^{da})T^{-da}(-d\lambda).\notag
\end{align}
Taking $T=\omega$ and then using Gross-Koblitz formula we obtain
\begin{align}\label{eq-5}
A&=\sum_{a=0}^{q-2}\pi^{(p-1)\sum_{i=0}^{r-1}\{d\langle\frac{ap^i}{q-1}\rangle+\langle\frac{-dap^i}{q-1}\rangle\}}
~\overline{\omega}^{da}(-d\lambda)\prod_{i=0}^{r-1}\Gamma_p^d\left(\left\langle\frac{ap^i}{q-1}\right\rangle\right)
\Gamma_p\left(\left\langle\frac{-dap^i}{q-1}\right\rangle\right).
\end{align}
Applying Lemma \ref{lemma4} for $t=d$ and $j=a$ we have
\begin{align}
\prod_{i=0}^{r-1}\Gamma_p\left(\left\langle\frac{-dap^i}{q-1}\right\rangle\right)=\omega^{da}(d)\prod_{i=0}^{r-1}
\frac{\prod_{h=1}^d\Gamma_p(\langle(\frac{h}{d}-\frac{a}{q-1})p^i\rangle)}
{\prod_{h=1}^{d-1}\Gamma_p(\langle\frac{hp^i}{d}\rangle)}.\notag
\end{align}
Substituting this into \eqref{eq-5} we obtain
\begin{align}
A&=\sum_{a=0}^{q-2}\pi^{(p-1)\sum_{i=0}^{r-1}\{-d\lfloor\frac{ap^i}{q-1}\rfloor-\lfloor\frac{-dap^i}{q-1}\rfloor\}}
~\overline{\omega}^{da}(-\lambda)\notag\\
&\times\prod_{i=0}^{r-1}\Gamma_p^d\left(\left\langle\frac{ap^i}{q-1}\right\rangle\right)
\frac{\prod_{h=1}^d\Gamma_p(\langle(\frac{h}{d}-\frac{a}{q-1})p^i\rangle)}
{\prod_{h=1}^{d-1}\Gamma_p(\langle\frac{hp^i}{d}\rangle)}\notag\\
&=\sum_{a=0}^{q-2}\pi^{(p-1)\sum_{i=0}^{r-1}\{-d\lfloor\frac{ap^i}{q-1}\rfloor-\lfloor\frac{-dap^i}{q-1}\rfloor\}}
~\overline{\omega}^{da}(-\lambda)\notag\\
&\times\prod_{i=0}^{r-1}\Gamma_p\left(\left\langle\frac{ap^i}{q-1}\right\rangle\right)\Gamma_p\left(\left\langle\left(1-\frac{a}{q-1}\right)p^i\right\rangle\right)
\Gamma_p^{d-1}\left(\left\langle\frac{ap^i}{q-1}\right\rangle\right)\notag\\
&\times \prod_{h=1}^{d-1}\frac{\Gamma_p(\langle(\frac{h}{d}-\frac{a}{q-1})p^i\rangle)}
{\Gamma_p(\langle\frac{hp^i}{d}\rangle)}.\notag
\end{align}
Upon simplification, we obtain
\begin{align}
A&=1+\sum_{a=1}^{q-2}(-p)^{\sum_{i=0}^{r-1}\{-d\lfloor\frac{ap^i}{q-1}\rfloor-\lfloor\frac{-dap^i}{q-1}\rfloor\}}
~\overline{\omega}^{da}(-\lambda)\notag\\
&\times\prod_{i=0}^{r-1}\Gamma_p\left(\left\langle\frac{ap^i}{q-1}\right\rangle\right)\Gamma_p\left(\left\langle\left(1-\frac{a}{q-1}\right)p^i\right\rangle\right)
\Gamma_p^{d-1}\left(\left\langle\frac{ap^i}{q-1}\right\rangle\right)\notag\\
&\times \prod_{h=1}^{d-1}\frac{\Gamma_p(\langle(\frac{h}{d}-\frac{a}{q-1})p^i\rangle)}
{\Gamma_p(\langle\frac{hp^i}{d}\rangle)}.\notag
\end{align}
From Lemma 3.4 of \cite{BSM}, for $0<a\leq q-2$ we have
$$\displaystyle\prod_{i=0}^{r-1}\Gamma_p\left(\left\langle\frac{ap^i}{q-1}\right\rangle\right)\Gamma_p\left(\left\langle\left(1-\frac{a}{q-1}\right)p^i\right\rangle\right)
=(-1)^r\overline{\omega}^a(-1).$$ 
Applying this in the above expression we deduce that
\begin{align}
A&=1+\sum_{a=1}^{q-2}(-1)^r(-p)^{\sum_{i=0}^{r-1}\{-d\lfloor\frac{ap^i}{q-1}\rfloor-\lfloor\frac{-dap^i}{q-1}\rfloor\}}
~\overline{\omega}^{da}(\lambda)\notag\\
&\times\prod_{i=0}^{r-1}\Gamma_p^{d-1}\left(\left\langle\frac{ap^i}{q-1}\right\rangle\right)
\prod_{h=1}^{d-1}\frac{\Gamma_p(\langle(\frac{h}{d}-\frac{a}{q-1})p^i\rangle)}
{\Gamma_p(\langle\frac{hp^i}{d}\rangle)}.\notag
\end{align}
Now, Lemma \ref{lemma5} yields
\begin{align}
A&=1+\sum_{a=1}^{q-2}(-1)^r(-p)^{\sum_{i=0}^{r-1}\{1-(d-1)\lfloor\frac{ap^i}{q-1}\rfloor
-\sum_{h=1}^{d-1}\lfloor\langle\frac{hp^i}{d}\rangle-\frac{ap^i}{q-1}\rfloor\}}~\overline{\omega}^{da}(\lambda)\notag\\
&\times\prod_{i=0}^{r-1}\Gamma_p^{d-1}\left(\left\langle\frac{ap^i}{q-1}\right\rangle\right)
\prod_{h=1}^{d-1}\frac{\Gamma_p(\langle(\frac{h}{d}-\frac{a}{q-1})p^i\rangle)}
{\Gamma_p(\langle\frac{hp^i}{d}\rangle)}.\notag
\end{align}
Adding and subtracting the term under summation for $a=0$ we obtain
\begin{align}
A&=1-q+q\sum_{a=0}^{q-2}(-p)^{\sum_{i=0}^{r-1}\{-(d-1)\lfloor\frac{ap^i}{q-1}\rfloor
-\sum_{h=1}^{d-1}\lfloor\langle\frac{hp^i}{d}\rangle-\frac{ap^i}{q-1}\rfloor\}}~\overline{\omega}^{da}(\lambda)\notag\\
&\times\prod_{i=0}^{r-1}\Gamma_p^{d-1}\left(\left\langle\frac{ap^i}{q-1}\right\rangle\right)
\prod_{h=1}^{d-1}\frac{\Gamma_p(\langle(\frac{h}{d}-\frac{a}{q-1})p^i\rangle)}
{\Gamma_p(\langle\frac{hp^i}{d}\rangle)}\notag\\
&=1-q-q(q-1){_{d-1}G}_{d-1}\left[\begin{array}{cccc}
                           \frac{1}{d}, & \frac{2}{d}, & \ldots, & \frac{d-1}{d} \\
                            0, & 0, & \ldots, & 0
                           \end{array}|\lambda^d
\right]_q.\notag
\end{align}
Finally, substituting the values of $A$ and $B$ into \eqref{eq-4} and then using \eqref{eq-6} we deduce the result. This completes the proof of Theorem \ref{MT_1}.
\bibliographystyle{amsplain}

\end{document}